\documentclass[notitlepage,english]{article}
\usepackage{amsmath,amssymb,amsfonts,amsthm,babel,graphicx}
\usepackage[T1]{fontenc}

\newtheorem{theorem}{Theorem}

\newtheorem{definition}{Definition}

\title{On the stratification by orbit types II}
\author{Julien Giacomoni \\
Institut de Math\'ematiques de Marseille - UMR7373 \\
39, rue F. Joliot Curie \\
13453 MARSEILLE Cedex 13 \\
julien.giacomoni@gmail.com}

\begin{document}

\maketitle




\begin{abstract}
When we have a proper action of a Lie group on a manifold, it is well known that we get a stratification by orbit types and it is known that this stratification satisfies the Whitney (b) condition. In a previous article we have seen that the stratification satisfies the strong Verdier condition. In this article we improve this result and obtain smooth local triviality. \\ \begin{center}MSC2010: 37C - 53C - 57R \end{center}
\end{abstract}

\begin{flushleft}

Smooth dynamical systems have been studied a lot and are still an active research domain. In particular, the study of symmetric dynamics, which are smooth dynamical systems that are symmetric (equivariant) with respect to a Lie group of transformation, may have great repercussions on mathematical physics and physics in general. The reader may consult "Dynamics and Symmetry" by M. Field (\cite{Fie}) and see the richness of this domain and the number of strong theoretical results that have been obtained. More fundamentally, proper actions of a Lie group on a manifold lead to nice examples of stratified spaces.


There are two levels of study when we consider a proper action of a Lie group on a manifold: stratification of the manifold itself and stratification of the quotient space called the orbit space. On both levels we get a stratification by orbit types with regularity conditions. The orbit space has been studied a lot and it is known that the stratification by orbit types is smoothly locally trivial (\cite{Sni}). This is stronger than the Whitney (b) condition or than the strong Verdier condition. Until now the only regularity condition for the stratification on the manifold that was given in the different references is the Whitney $(b)$ condition, which is a generic condition (in the sense that every algebraic variety or semi-algebraic set admits a Whitney stratification) and recently the strong Verdier stratification which is non-generic (\cite{Gia}). In this article we will see that the orbit type stratification on the manifold is smoothly locally trivial.

The new regularity obtained here is stronger than those previously obtained and may have nice repercussions. In a certain sense these results answer the question asked implicitely by Duistermaat and Kolk in \cite{DuKo} in (2.7.5): "We feel that the stratification by orbit types has even more special properties than general Whitney stratifications".

Independantly of this work, we can notice that recently the same result has been proposed by C.T.C. Wall in \cite{Wal}.

\section{Stratification by orbit types}

In this section we will recall the definitions and principal results about stratifications by orbit type. The classical references underlying what follows are \cite{Ko}, \cite{Pal}. We will follow mostly the notations of M. Pflaum in \cite{Pfl} which synthesizes work \cite{Bre}, \cite{DoS}, \cite{Fer}, \cite{Jan}, \cite{Les}, \cite{Sj}.

Let $\mathcal{M}$ be a manifold and $G$ a Lie group. 
\begin{definition}
A (left) action of $G$ is a smooth mapping (i.e. $\mathcal{C}^{\infty}$) $$\Phi : G \times \mathcal{M} \rightarrow \mathcal{M}, (g,x) \mapsto \Phi(g,x)=\Phi_{g}(x)=gx$$ such that: $$\forall g,h \in G, \forall x \in \mathcal{M}, \Phi_{g}(\Phi_{h}(x))=\Phi_{gh}(x),$$ and $\Phi_{e}(x)=x$, where $e$ is the unit element of $G$.\\
\end{definition}

\begin{definition}
A $G$-action $\Phi: G \times \mathcal{M} \rightarrow \mathcal{M}$ is called \emph{proper} if the mapping $\Phi_{ext}: G \times \mathcal{M} \rightarrow \mathcal{M} \times \mathcal{M}$, $(g,x) \mapsto (gx,x)$ is proper.
\end{definition}

With such proper actions several results are known, in particular $\mathcal{M}$ admits a $G$-invariant Riemannian metric. The most important result is the so called slice theorem (\cite{Ko}, \cite{Pal}). Here it is as stated in \cite{Pfl}:

\begin{theorem}
Let $\Phi: G \times \mathcal{M} \rightarrow \mathcal{M}$ be a proper group action, $x$ a point of $\mathcal{M}$ and $\mathcal{V}_{x}=T_{x}\mathcal{M}/T_{x}Gx$ the normal space to the orbit of $x$. Then there exists a $G$-equivariant diffeomorphism from a $G$-invariant neighborhood of the zero section of $G \times_{G_{x}} \mathcal{V}_{x}$ onto a $G$-invariant neighborhood of $Gx$ such that the zero section is mapped onto $Gx$ in a canonical way (where $G_{x}$ is the isotropy group of $x$).
\end{theorem}

If we denote by $\mathcal{M}_{(H)}$ the set $\{x\ \in \mathcal{M} | G_{x} \sim H\}$ where $\sim$ means "conjugate to", we get in particular that for a compact subgroup $H$ of $G$ each connected component of $\mathcal{M}_{(H)}$ is a submanifold of $\mathcal{M}$. The isotropy subgroups $G_{x}$ are compact in the case of a proper group action. Assigning to each point $x \in \mathcal{M}$ the germ $\mathcal{S}_{x}$ of the set $\mathcal{M}_{(G_{x})}$ we get a stratification of $\mathcal{M}$ in the sense of Mather (\cite{Mat}), called stratification by orbit type.

This stratification has been studied a lot and has been also recently described in \cite{DuKo}, \cite{Fie}. This stratification was known to be Whitney $(b)$ regular and is known to satisfy the strong Verdier condition (\cite{Gia}):

\begin{definition}
Let $X$ be a $\mathcal{C}^{1}$ submanifold of $\mathbb{R}^{n}$. Let $Y$ be a submanifold of $\mathbb{R}^{n}$ such that $0 \in Y \subset \overline{X} \setminus X$. In \cite{LKT} (see also \cite{LTrW}) Kuo, Li, Trotman and Wilson define $X$ to be strongly Verdier regular over $Y$ (or differentiably regular) at $0$ if for all $\epsilon>0$ there is a neighborhood $U$ of $0$ in $\mathbb{R}^{n}$ such that if $x \in U \cap X$ and $y \in U \cap Y$, then $$d(T_{y}Y,T_{x}X)\leqslant \epsilon|x-y|.$$
\end{definition}

\begin{theorem}
The stratification by orbit types of a $G$-manifold $\mathcal{M}$ with a proper action is a strong Verdier stratification.
\end{theorem}

\section{Smooth local triviality}

Let us look at the definition of the strongest condition that we may expect for a stratification.

\begin{definition}
A stratified space $X$ is called smoothly locally trivial if for every $x \in X$ there exists a neighborhood $\mathcal{U}$, a stratified space $F$ with stratification $\mathcal{S}^{F}$, a distinguished point $o \in F$ and a smooth isomorphism of stratified spaces $h:\mathcal{U} \rightarrow (\mathcal{S}\cap \mathcal{U}) \times F$ such that $h^{-1}(y,o)=y ,\forall y \in \mathcal{S} \cap \mathcal{U}$ and such that $\mathcal{S}^{F}_{o}$ is the germ of the set ${o}$. Here, $\mathcal{S}$ is the stratum of $X$ with $x \in S$.
\end{definition}

\section{Theorem}
\begin{theorem}
The stratification by orbit types of a $G$-manifold $\mathcal{M}$ with a proper action is smoothly locally trivial.
\end{theorem}

\begin{proof}

Let us begin with two subgroups of G.\\

Suppose that $K \subsetneq H \subset G$ are two isotropy groups of $\mathcal{M}$, so that we have $\mathcal{M}_{(H)}<\mathcal{M}_{(K)}$. Let $y \in \mathcal{M}_{(H)}$. By the slice theorem, we can suppose that:
$$\mathcal{M}=G \times_{H} \mathcal{V}=(G \times_{H} \mathcal{W}) \times \mathcal{V}^{H}$$ and $y=[(e,0)]$ where $\mathcal{V}$ is an $H$-slice, $\mathcal{V}^{H}$ is the subspace of the $H$-invariant vectors, and $\mathcal{W}=(\mathcal{V}^{H})^{\perp}$ is the orthogonal space relative to the $H$-invariant inner product on $\mathcal{V}$.\\

We have (\cite{Pfl} page 159): $$\mathcal{M}_{(K)}=(G \times_{H} \mathcal{W}_{(K)}) \times \mathcal{V}^{H}$$ and $$\mathcal{M}_{(H)}=G/H \times \{0\} \times \mathcal{V}^{H}$$.\\

Let us consider $H \subset G$ an isotropy group of $\mathcal{M}$ and let $y \in \mathcal{M}_{(H)}$. Possibly restricting the open set $\mathcal{U}$ around $y$ we have a finite number of strata to consider, the strata such that $\mathcal{M}_{(H)}$ is in their boundaries, using that we we have a Whitney stratification (\cite{Pfl}). If we look at the previous considerations, they are independant of the subgroup $K$ describing the strata around $y \in \mathcal{M}_{(H)}$. So, we have a $G$-equivariant $\mathcal{C}^{\infty}$-diffeomorphism $\phi$ in Koszul's structural theorem (Theorem 1) transforming locally the stratified space $\mathcal{M}$: $$ \phi: \mathcal{U} \rightarrow G \times_{G_{y}}\mathcal{V}_{y}.$$ With the previous notations: for $x \in \mathcal{M}_{(H)}$, we have:
$$\tilde{\phi}: \mathcal{U} \rightarrow \mathcal{U} \cap (G/H \times \mathcal{V}^{H}) \times \cup_{(H)\leqslant(K)}G \times_{H}\mathcal{W}_{(K)}$$ with $\tilde{\phi}([g,v])=((gH,v^{H}),[g,w])$ (where $v=(v^{H},w)$) which is a smooth isomorphism of stratified spaces such that: $$\tilde{\phi}^{-1}(h,o)=h, \forall h \in \mathcal{U} \cap (G/H \times \mathcal{V}^{H})$$ where $$F=\cup_{(H)\leqslant(K)}G \times_{H}\mathcal{W}_{(K)}$$ and such that $\mathcal{S}^{F}_{o}$ is the germ of ${o}$ where $o$ is $[e,0]$. This shows the result.
\end{proof}

I would like to thank David Trotman for his precious encouragement.

\end{flushleft}

\end{document}